\documentclass{amsart}
\usepackage{tikz}
\usetikzlibrary{shapes, arrows}
\usepackage{amssymb}
\usepackage{hyperref}
\usepackage{tikz}

\newtheorem{theorem}{Theorem}[section]

\newtheorem{cor}[theorem]{Corollary}

\theoremstyle{definition}
\newtheorem{definition}[theorem]{Definition}
\newtheorem{example}[theorem]{Example}

\theoremstyle{remark}

\numberwithin{equation}{section}

\usepackage{url}

\begin{document}
%
\title{Vulnerability Measures and Zagreb Indices of Graphs }
\author{Sanju Vaidya}
\address{Math \& CS Department, Mercy University, 555 Broadway, Dobbs Ferry, NY 10522}
\curraddr{}
\email{SVaidya@mercy.edu}
\thanks{}

\author{Jeff Chang}
\address{Math \& CS Department, Mercy University, 555 Broadway, Dobbs Ferry, NY 10522}
\curraddr{}
\email{cchang4@mercy.edu}
\thanks{}

%
%

%


\begin{abstract}
 This paper establishes sharp bounds for the vulnerability measures of closeness and generalized closeness in graphs and identifies graphs that attain these bounds. It further develops bounds incorporating Zagreb indices for triangle- and quadrangle-free graphs, yielding formulas for closeness and generalized closeness in such graphs with diameter at most 3. Moreover, using Zagreb indices, we derive bounds for trees and connected graphs with girth at least 7, which are attained by graphs with diameter at most 4. Finally, formulas for closeness and generalized closeness in specific trees are established using Zagreb indices.

\end{abstract}

\maketitle              

\section{Introduction}
It is widely acknowledged that cybersecurity is of paramount importance in our daily lives. Over the past four decades, graph theory has provided numerous tools for analyzing communication networks. In this study, network processors are represented by vertices, and the connections between them are represented by edges. Various graph vulnerability measures are critical for assessing the stability and reliability of these communication networks. Dangalchev \cite{dangalchev2006residual, dangalchev2011residual} introduced novel vulnerability measures, closeness and generalized closeness, demonstrating their superior sensitivity compared to other measures such as graph toughness and integrity. For further research, please refer to \cite{aytac2011residual, berberler2017residual, aytacc2018network, aytacc2018closeness, dangalchev2023closeness, dangalchev2024link, golpek2024closeness, golpek2023computing, rupnik2025generalized}.

The primary research question focuses on determining the bounds for various graph vulnerability measures. This paper will derive upper and lower bounds for the vulnerability measures, closeness, and generalized closeness of graphs. We will demonstrate that these bounds are achieved by any graph with a diameter of at most 2.

Furthermore, we will determine bounds for triangle- and quadrangle-free graphs using Zagreb indices. These indices, initially introduced by Gutman-Trinajstic \cite{gutman1972graph} for chemical applications, have been extensively studied for over 30 years \cite{gutman2018beyond}. The first Zagreb index represents the sum of the squares of the vertex degrees, while the second Zagreb index represents the sum of the products of the degrees of adjacent vertex pairs. We will demonstrate that the bounds involving the first  Zagreb index are achieved by any triangle- and quadrangle-free graph with a diameter of at most 3. Moreover, utilizing the bounds established by Yamaguchi S. \cite{yamaguchi2008estimating}, we will derive an upper bound for closeness in terms of the number of vertices,  edges, and radius which is achieved by a Moore graph of diameter 2 and a cycle of length 6. In addition, we will determine bounds involving the first and second Zagreb indices for trees and connected graphs with girth at least 7. Further, we will show that these bounds are attained by such graphs of diameter at most 4. Additionally, we will derive formulas for the vulnerability measures, closeness, and generalized closeness for certain special trees $T(n, D)$, which were introduced by Furtula-Gutman-Ediz \cite{furtula2014difference} to analyze the difference between the first and second Zagreb indices.

Section 2 will review vulnerability measures and Zagreb indices of graphs. Section 3 will establish bounds for the vulnerability measures, and Section 4 will present a discussion and conclusion.

\section{Review of Graph Vulnerability Measures and Zagreb Indices}
This section provides a comprehensive review of graph vulnerability measures, including closeness, generalized closeness, and Zagreb indices. Let $G$ be a connected graph with $n$ vertices and $m$ edges. Let $V = V(G)$ denote the vertex set and $E = E(G)$ denote the edge set of graph $G$. Let $d_u$ represent the degree of vertex $u$, and $d(i, j)$ represent the distance between vertices $i$ and $j$. The eccentricity $e(v)$ of a vertex $v$ is defined as $e(v) = \max_{w\in V} d(v, w)$. The radius $r = r(G)$ and the diameter $d = d(G)$ are defined as $r = \min_{v \in V} e(v)$ and $d = \max_{v \in V} e(v)$, respectively. For any positive integer $k$, let $d(G, k)$ denote the number of pairs of vertices at distance $k$.

To evaluate the stability of complex networks, Dangalchev \cite{dangalchev2006residual} introduced the following vulnerability measure, closeness.

\begin{definition}
The closeness, $C(G)$, of a graph $G$ is defined as:
\begin{equation*}
C(G) = \sum_{i=1}^n C(i), \, C(i) = \sum_{j\not= i} \frac{1}{2^{d(i, j)}}
\end{equation*}
\end{definition}

Dangalchev \cite{dangalchev2011residual} extended the closeness measure to generalized closeness by utilizing a real number $\alpha$ such that $0 < \alpha < 1$ instead of $1/2$.

\begin{definition}
The generalized closeness, $GC(G)$, of a graph $G$ is defined as:
\begin{equation*}
GC(G) = \sum_{i=1}^n GC(i), \, GC(i) = \sum_{j\not= i} \alpha^{d(i, j)}, 0<\alpha<1
\end{equation*}
\end{definition}

Gutman and Trinajstic \cite{gutman1972graph} introduced the Zagreb indices, which are defined as follows.

\begin{definition}
The first and second Zagreb indices of a graph $G$ are defined as:
\begin{equation*}
M_1 = M_1(G) = \sum_{v\in V(G)} d_{v}^2 , M_2 = M_2(G) = \sum_{uv\in E(G)} d_ud_{v}
\end{equation*}
\end{definition}

Furtula, Gutman, and Ediz \cite{furtula2014difference} introduced the following reduced second Zagreb index.

\begin{definition}

The reduced second Zagreb index of a graph $G$ is defined as:

\begin{equation*}
RM_2 = RM_2(G) = \sum_{uv\in E(G)} (d_u - 1)(d_{v} - 1)
\end{equation*}

\end{definition}

The following result by Gutman \cite{gutman1997property} is necessary.

\begin{theorem}[Gutman]

Let $G$ be a connected acyclic graph with $n$ vertices and diameter $d$. Let $\phi$ be a graph invariant defined as $\phi = \phi(G) = \sum_{u < v} f(d(u, v))$, where $d(u, v)$ is the distance between vertices $u$ and $v$. If the function $f(x)$ monotonically decreases for $x \in [1, d]$, then $\phi(G) \ge \phi(P_n)$, where $P_n$ is the path with $n$ vertices.

\end{theorem}

The following result of Gutman and Das \cite{gutman2004first} is also required.

\begin{theorem}

Let $G$ be a triangle- and quadrangle-free graph with $n$ vertices and $m$ edges. Let $M_1$ be the first Zagreb index of the graph $G$. Then for the number of pairs of vertices at distance 2, $d(G,2)$, we have:

\begin{equation*}
d(G, 2)= 0.5M_1 - m
\end{equation*}

\end{theorem}

The following result of Yamaguchi S. \cite{yamaguchi2008estimating} is also required. The bounds in the result are achieved by any Moore graph of diameter 2. Examples of Moore graphs of diameter 2 from \cite{hoffman1960moore} are the pentagon and Petersen graph.

\begin{theorem}\label{T:Yamaguchi}

Let $G$ be a triangle- and quadrangle-free connected graph with $n$ vertices and radius $r$. Let $M_1$ be the first Zagreb index of graph $G$. Then
\begin{equation*}
M_1(G) \le n(n + 1 - r)
\end{equation*}
and equality holds if and only if graph $G$ is a Moore graph of diameter 2 or is isomorphic to $C_6$, where $C_6$ is a cycle of length 6.

\end{theorem}

Note that the Wiener polarity index $W_p(G)$ is defined as the number of pairs of vertices at distance 3. We will need the following result of Liu M. Liu B \cite{liu2011wiener} regarding the Wiener Polarity index, Zagreb indices, and the girth $g(G)$ of a connected graph $G$ (the length of the shortest cycle in $G$),

\begin{theorem}

Let $G$ be a connected $(n,m)$ graph. Then
\begin{equation*}
W_P(G) \le M_2(G) - M_1(G) + m
\end{equation*}
with equality holding if and only if the graph $G$ is a tree or its girth, $g(G) \ge 7$.

\end{theorem}

\section{Bounds for closeness and generalized closeness}

In this section, we will find upper and lower bounds for the vulnerability measures, closeness and generalized closeness, of graphs. 

\begin{theorem}

Let G be a graph with n vertices, m edges, and diameter d. Then for the generalized closeness, GC(G), and for the closeness, C(G),we have

\begin{equation*}
\frac{2[n\alpha(1-\alpha)-\alpha(1-\alpha^n)]}{(1-\alpha)^2} \le GC(G) \le \alpha n(n-1).
\end{equation*}

\begin{equation*}
2n - 4 +(0.5)^{n-2} \le C(G) \le \frac {n(n-1)}{2}.
\end{equation*}

Moreover, for each formula, the left equality occurs if G is a path, $P_n$, and the right equality occurs if G is a complete graph. 

\end{theorem}

\begin{proof}
We have

\begin{equation*}
GC(G) = \sum_i \sum_ {j\not= i} \alpha^ {d(i,j)} = 2\sum_{i<j} \alpha^ {d(i,j)}.
\end{equation*}

If we add an edge to the graph $G$, then the distance $d(i,j)$ will decrease and the generalized closeness, $GC(G)$, will increase. 

So the generalized closeness, $GC(G)$, will be maximum if $G$ is a complete graph and is minimum if it is tree. If $G$ is a complete graph, then

$GC(G) = \frac{2n(n-1)\alpha}{2} = \alpha n(n-1)$. Additionally, by Theorem 2.5 of Gutman \cite{gutman1997property}, we have $GC(P_n) \le GC(G)$ if $G$ is a tree, since the function $f(x) = \alpha^x$ monotonically decreases for $x\in [1, d]$. Note that $GC(P_n) = 2\sum_{k=1}^{n-1} (n-k)\alpha^k$. To compute this sum, we note the following: $\sum_{i=0}^{n-1} x^i = \frac{x^n - 1}{x - 1}$. By differentiating, we get $\sum_{i=1}^{n-1} i x^{i-1} = \frac{(n-1)x^n - nx^{n-1} + 1}{(x - 1)^2}$. Let $x = \frac{1}{\alpha}$, multiply by $\alpha^{n-1}$, and simplify. This will yield the result for generalized closeness $GC(G)$. The result for closeness $C(G)$ follows from it by setting $\alpha = 0.5$. 

\end{proof}

Aytac and Odabas Berberler \cite{aytacc2018network} developed a formula for closeness when the graph's diameter is less than or equal to 2. In the following theorem, we will establish bounds for generalized closeness and closeness, which are achieved by graphs with diameters less than or equal to 2. We will obtain the same formula as presented in Theorem 4.1 of Aytac and Berberler \cite{aytacc2018network} for closeness if the diameter is less than or equal to 2.

\begin{theorem}

Let $G$ be a graph with $n$ vertices, $m$ edges, and diameter $d$. Then, for the generalized closeness, $GC(G)$, and for the closeness, $C(G)$, we have

\begin{equation*}
\alpha^d n(n-1) + 2m\alpha(1 - \alpha^{d-1}) \le GC(G) \le \alpha^2 n(n-1) + 2m\alpha(1 - \alpha).
\end{equation*}

\begin{equation*}
\frac {n(n-1)}{2^d} + m(1 - 0.5^{d-1}) \le C(G) \le \frac {n(n-1)+2m}{4}.
\end{equation*}
where for each formula, equality holds if the diameter $d$ is less than or equal to 2.

\end{theorem}

\begin{proof}

Let $V(G)$ be the set of all vertices of the graph $G$ and $N_G(i)$ be the set of all vertices in the neighborhood of vertex $i$.

We have

\begin{equation*}
GC(i) = \sum_{j\in V(G), j\not= i} \alpha^ {d(i,j)} = \sum_{j\in N_G(i), j\not=i} \alpha^ {d(i,j)} + \sum_{j\in V(G)\setminus N_G(i), j\not=i} \alpha^ {d(i,j)}
\end{equation*}

\begin{equation*}
GC(i) = \alpha deg_G(i) + \sum_{j\in V(G)\setminus N_G(i)} \alpha^ {d(i,j)}
\end{equation*}

\begin{equation*}
GC(i) \le \alpha deg_G(i) + \alpha^2(n - 1 - deg_G(i))
\end{equation*}

Now, the generalized closeness, $GC(G) = \sum_{i=1}^n GC(i)$. By substituting for $GC(i)$, we get

\begin{equation*}
GC(G) \le \sum_{i=1}^n \alpha deg_G(i) + \alpha^2(n - 1 - deg_G(i))
\end{equation*}

\begin{equation*}
GC(G) \le \sum_{i=1}^n \alpha^2(n - 1) + \alpha (1 - \alpha) deg_G(i)
\end{equation*}

\begin{equation*}
GC(G) \le \alpha^2 n(n - 1) + \alpha (1 - \alpha)(2m)
\end{equation*}
This yields the upper bound. To obtain the lower bound, note the following.

\begin{equation*}
GC(i) \ge \alpha deg_G(i) + \alpha^d(n - 1 - deg_G(i))
\end{equation*}

\begin{equation*}
GC(G) \ge \sum_{i=1}^n \alpha deg_G(i) + \alpha^d(n - 1 - deg_G(i))
\end{equation*}

\begin{equation*}
GC(G) \ge \sum_{i=1}^n \alpha^d(n - 1) + \alpha (1 - \alpha^{d-1}) deg_G(i)
\end{equation*}

\begin{equation*}
GC(G) \ge \alpha^d n(n - 1) + \alpha (1 - \alpha^{d-1})(2m)
\end{equation*}

This gives the lower bound. Furthermore, it is readily apparent that equality holds for both bounds if the diameter d is less than or equal to 2. This result for closeness, $C(G)$, follows by setting $\alpha = 0.5$.

\end{proof}

The following theorem establishes bounds for the generalized closeness, $GC(G)$, and closeness, $C(G)$, for any triangle- and quadrangle-free graph. These bounds yield formulas for $GC(G)$ and $C(G)$ for such graphs when the diameter d is less than or equal to 3.

\begin{theorem}\label{T:TQ}
Let G be a triangle- and quadrangle-free graph with n vertices, m edges, and diameter d. Let $M_1$ be the first Zagreb index of G. Let

\[L= \alpha^d( n(n-1)- M_1) + \alpha^2(M_1 - 2m) + 2m\alpha\]
and
\[U= \alpha^3( n(n-1)- M_1) + \alpha^2(M_1 - 2m) + 2m\alpha\]
Then, for the generalized closeness, GC(G), we have $L \le GC(G) \le U$ and for the closeness, C(G), we have

\[\frac{ n(n-1)- M_1}{2^d} + \frac{M_1 + 2m}{4} \le C(G) \le \frac{n(n - 1) + M_1 + 4m}{8}\]
where for each formula, equality holds if the diameter d is less than or equal to 3.

\end{theorem}
\begin{proof}
For any positive integer $k$, let $d(G, k)$ denote the number of vertex pairs at distance $k$. Gutman and Das \cite{gutman2004first} demonstrated that $d(G, 2)= 0.5M_1 - m$. We utilize this to derive the bounds as follows:

\[GC(G) = 2[\alpha m + \alpha^2(0.5M_1 - m) + \sum_{i<j, d(i,j)\ge 3} \alpha^ {d(i,j)}]\]

Since $0 < \alpha < 1$, we obtain

\[GC(G) \ge 2[\alpha m + \alpha^2(0.5M_1 - m) + \alpha^d( 0.5n(n-1)- 0.5M_1)]\]

Simplifying yields the lower bound $L$. Similarly,

\[GC(G) \le 2[\alpha m + \alpha^2(0.5M_1 - m) + \alpha^3( 0.5n(n-1) - 0.5M_1)]\]

Simplifying yields the upper bound $U$. Equality holds for both bounds if the diameter $d$ is less than or equal to 3. The result for closeness follows by setting $\alpha = 0.5$.
 
\end{proof}

In the following corollary, we will establish a precise upper bound for the closeness of triangle- and quadrangle-free graphs. The bound will depend on the number of vertices, edges, and radius.

\begin{cor}

Let $G$ be a triangle- and quadrangle-free graph with $n$ vertices, $m$ edges, and radius $r$. Then

\begin{equation*}
C(G) \le \frac{n(2n - r) + 4m}{8}
\end{equation*}
with equality if the graph $G$ is a Moore graph of diameter 2 or a cycle of length 6, $C_6$.

\end{cor}

\begin{proof}

This follows from Theorem \ref{T:TQ} and Theorem \ref{T:Yamaguchi}.

\end{proof}

In the following theorem, we will determine bounds for the generalized closeness, $GC(G)$, and the closeness, $C(G)$, for any graph $G$ that is a tree or has a girth (the length of the shortest cycle in $G$), $g(G) \ge 7$. This will provide formulas for the generalized closeness, $GC(G)$, and the closeness, $C(G)$, for any such graph if the diameter $d$ is less than or equal to 4.

\begin{theorem}\label{T:Zagreb}

Let $G$ be a tree or a connected graph with girth $g(G) \ge 7$. Let the graph $G$ have $n$ vertices, $m$ edges, and diameter $d$. Let $M_1$ and $M_2$ be the Zagreb indices of the graph $G$. Let
\begin{equation*}
L= \alpha^d( n(n-1)+ M_1 - 2m -2M_2) + 2\alpha^3(M_2 + m)+\alpha^2M_1(1-2\alpha) + 2m\alpha(1 - \alpha)
\end{equation*}
and
\begin{equation*}
U= \alpha^4( n(n-1)+ M_1 - 2m -2M_2) + 2\alpha^3(M_2 + m)+\alpha^2M_1(1-2\alpha) + 2m\alpha(1 - \alpha)
\end{equation*}
Then for the generalized closeness, $GC(G)$, we have
\begin{equation*}
L \le GC(G) \le U.
\end{equation*}
For the closeness, $C(G)$, we have
\begin{equation*}
\frac{ n(n-1)+ M_1 - 2m -2M_2}{2^d} + \frac{M_2 + 3m}{4} \le C(G) \le \frac{n(n - 1) + M_1 +2M_2 + 10m}{16}
\end{equation*}
where for each formula, equality holds if the diameter $d$ is less than or equal to 4.

\end{theorem}

\begin{proof}

Let $d(G, k)$ be the number of pairs of vertices at distance $k$, $k>0$. Then, by Theorem 2.7 of Liu, Liu \cite{liu2011wiener} we have

\begin{equation*}
d(G, 3) = M_2 -M_1 + m.
\end{equation*}
By Theorem (2.6) we also have $d(G, 2)= 0.5M_1 - m$. We will use this to obtain the bounds as follows. We have
\begin{equation*}
GC(G) = 2[\alpha m + \alpha^2(0.5M_1 - m) + \alpha^3(M_2 - M_1 + m) + \sum_{i<j, d(i,j)>3} \alpha^ {d(i,j)}]
\end{equation*}
Since $0 < \alpha < 1$, we obtain
\begin{align*}
GC(G) \ge 2[\alpha m + \alpha^2(0.5M_1 - m) &+ \alpha^3(M_2 - M_1 + m) +\\
&\alpha^d( 0.5n(n-1)+ 0.5M_1 - m -M_2)].
\end{align*}
Simplifying this gives the lower bound $L$. Similarly, we have
\begin{align*}
GC(G) \le 2[\alpha m + \alpha^2(0.5M_1 - m) &+ \alpha^3(M_2 - M_1 + m)  +\\
&\alpha^4( 0.5n(n-1)+ 0.5M_1 - m -M_2)].
\end{align*}
Simplifying this gives the upper bound $U$. Furthermore, it is readily apparent that equality for both bounds occurs if the diameter $d$ is less than or equal to 4. Also, the result for closeness follows by setting $\alpha = 0.5$.

\end{proof}

In Theorem 9, Dangalchev C. \cite{dangalchev2023closeness} presented a formula for closeness in bistar graphs. Furthermore, Hande Tuncel Golpek and Aysun Aytac \cite{golpek2023computing} also presented a formula for bistar graphs. Furtula-Gutman-Ediz \cite{furtula2014difference} defined the following class of trees, which is a generalization of bistar graphs.

\begin{definition}

Let D be a positive integer, $D\ge 2$. Let $S_{D+1}$ be the star on D+1 vertices, and let $v_1, v_2, ..., v_D$ be its pendent vertices. For $i = 1, 2, ..., D$ let $r_i$ be non-negative integers such that $r_1 \ge r_2 \ge ... \ge r_D$. Construct the tree $T(r_1\, r_2, ..., r_D)$ by attaching $r_i$ pendent vertices to the vertex $v_i$ of $S_{D+1}$, and by doing this for $i = 1, 2, ..., D$. The tree $T(r_1\, r_2, ..., r_D)$ has $ n = 1 + D + \sum_{i=1}^D r_i$ vertices. For given values of $D\ge 2$ and $n\ge D + 1$, the set of all trees $T(r_1\, r_2, ..., r_D)$ constructed in the above manner is denoted by $T(n, D)$.

\end{definition}

\begin{example}

Figures \ref{fig:1} and \ref{fig:2} show trees $T(n, D)$ from Furtula-Gutman-Ediz \cite{furtula2014difference}, where $n = 10$ and $D = 4$. In Figure 1, $ r_1 = 5$ and $r_2 = r_3 = r_4 = 0$. It is a bistar graph with diameter 3. The closeness is 23.25. In Figure 2, $ r_1 = 4, r_2 = 1$ and $r_3 = r_4 = 0$. The diameter is 4 and the closeness is 21.75.

\end{example}

\begin{figure}[h]
\centering
\begin{minipage}{.45\textwidth}
  \centering
  \begin{tikzpicture}[scale=.6,thick, vertex/.style={scale=.6, ball color=black, circle, text=white, minimum size=.2mm},
Ledge/.style={to path={
.. controls +(45:2) and +(135:2).. (\tikztotarget) \tikztonodes}}
]
\node [ label={[label distance=.1cm]90:$ $}]  (1) at (0, 0) [vertex] {};
\node [ label={[label distance=.1cm]90:$ $}]  (2) at (-1, 1) [vertex] {};
\node [ label={[label distance=.1cm]90:$ $}]  (4) at (1, 1) [vertex] {};
\node [ label={[label distance=.1cm]90:$ $}]  (6) at (0, 1) [vertex] {};
\node [ label={[label distance=.1cm]90:$ $}]  (7) at (0, 2) [vertex] {};
\node [ label={[label distance=.1cm]90:$ $}]  (8) at (-.6, 3) [vertex] {};
\node [ label={[label distance=.1cm]90:$ $}]  (9) at (0, 3) [vertex] {};
\node [ label={[label distance=.1cm]90:$ $}]  (10) at (.6, 3) [vertex] {};
\node [ label={[label distance=.1cm]90:$ $}]  (11) at (-1.4, 3) [vertex] {};
\node [ label={[label distance=.1cm]90:$ $}]  (12) at (-.6, 3) [vertex] {};
\node [ label={[label distance=.1cm]90:$ $}]  (13) at (1.4, 3) [vertex] {};
\draw (6) -- (2);
\draw (6) -- (4);
\draw (1) -- (6);
\draw (7) -- (6);
\draw (7) -- (8);
\draw (9) -- (7);
\draw (10) -- (7);
\draw (11) -- (7);
\draw (12) -- (7);
\draw (13) -- (7);
\end{tikzpicture}
\caption{}
  \label{fig:1}
\end{minipage}%
\begin{minipage}{.45\textwidth}
  \centering
  \begin{tikzpicture}[scale=.6,thick, vertex/.style={scale=.6, ball color=black, circle, text=white, minimum size=.2mm},
Ledge/.style={to path={
.. controls +(45:2) and +(135:2).. (\tikztotarget) \tikztonodes}}
]
\node [ label={[label distance=.1cm]90:$ $}]  (16) at (5, -1) [vertex] {};
\node [ label={[label distance=.1cm]90:$ $}]  (17) at (5, 0) [vertex] {};
\node [ label={[label distance=.1cm]90:$ $}]  (18) at (4, 1) [vertex] {};
\node [ label={[label distance=.1cm]90:$ $}]  (20) at (6, 1) [vertex] {};
\node [ label={[label distance=.1cm]90:$ $}]  (22) at (5, 1) [vertex] {};
\node [ label={[label distance=.1cm]90:$ $}]  (23) at (5, 2) [vertex] {};
\node [ label={[label distance=.1cm]90:$ $}]  (24) at (4.4, 3) [vertex] {};
\node [ label={[label distance=.1cm]90:$ $}]  (27) at (3.6, 3) [vertex] {};
\node [ label={[label distance=.1cm]90:$ $}]  (29) at (6.4, 3) [vertex] {};
\node [ label={[label distance=.1cm]90:$ $}]  (30) at (5.6, 3) [vertex] {};
\node [ label={[label distance=.1cm]90:$ $}]  (34) at (5.6, 3) [vertex] {};

\draw (16) -- (17);
\draw (22) -- (18);
\draw (22) -- (20);
\draw (17) -- (22);
\draw (23) -- (22);
\draw (23) -- (24);
\draw (27) -- (23);
\draw (29) -- (23);
\draw (30) -- (23);
\end{tikzpicture}
\caption{}
  \label{fig:2}
\end{minipage}
\end{figure}

In the following corollary, we will present formulas for generalized closeness and closeness for trees in $T(n, D)$.

\begin{cor}

Let $G$ be a tree $T(r_1, r_2, ..., r_D) \in T(n, D)$ with $n > D + 1$ vertices. Let $M_1$ and $M_2$ be the Zagreb indices of the graph $G$. If $r_2 = ... = r_D = 0$, then the generalized closeness is given by
\begin{equation*}
GC(G) = 2\alpha^3(D - 1)(n - D - 1) + \alpha^2M_1 + 2(n - 1)\alpha(1 - \alpha)
\end{equation*}
and the closeness is given by
\begin{equation*}
C(G) = \frac{(D - 1)(n - D - 1) + M_1 + 2(n - 1)}{4}.
\end{equation*}

If $r_1, r_2 > 0$, then the generalized closeness is given by
\begin{align*}
GC(G) = \alpha^4[n(n - 1) - M_1] + 2\alpha^3(D - 1)(n& - D - 1)(1 - \alpha) +\\
&\alpha^2M_1 + 2(n - 1)\alpha(1 - \alpha)
\end{align*}
and the closeness is given by
\begin{equation*}
C(G) = \frac{(n - 1)(n + 8) + 2(D - 1)(n - D - 1) + 3M_1}{16}.
\end{equation*}
\end{cor}

\begin{proof}

Furtula, Gutman, and Ediz \cite{furtula2014difference} demonstrated that if $r_2 = ... = r_D = 0$, then the diameter of graph $G$ is 3, and if $r_1, r_2 > 0$, then the diameter of graph $G$ is 4. Furthermore, they showed that the reduced second Zagreb index $RM_2(G) = M_2 - M_1 + n - 1 = (D - 1)(n - D - 1)$. Using this and Theorem \ref{T:Zagreb}, we obtain the results.

\end{proof}

\section{Discussion and Conclusion}
Graph theory offers robust tools for addressing diverse challenges, including the stability assessment of communication networks. Our primary research objective is the computation of various graph vulnerability measures. This paper presents bounds and formulas for the vulnerability measures, closeness and generalized closeness, across various graph types. Specifically, we derive formulas for triangle- and quadrangle-free graphs with diameters of four or less, and certain specialized trees, utilizing Zagreb indices—widely applied in chemical applications. 

The fields of communication network modeling and chemical graph theory present significant research opportunities. Given the numerous vulnerability measures and topological indices, the derivation of bounds and formulas for these remains an active area of research, crucial for network stability assessment and the modeling of diverse chemical compound properties.

\bibliographystyle{vancouver}
\bibliography{VancouverExamples.bib}

%
%
%










\end{document}